\documentclass[12pt]{article}

\usepackage[hidelinks, colorlinks]{hyperref}
\usepackage{amsmath,amsthm, amssymb, enumitem, microtype}
\DeclareRobustCommand \d{%
  \ifmmode \mathinner{\mathrm d}\mathchoice{\mskip -\thinmuskip }{\mskip -\thinmuskip }{}{}{}\else\fi
}
\allowdisplaybreaks
\theoremstyle{definition}

\theoremstyle{plain}
\newtheorem{theorem}{Theorem}[section]

\newtheorem{lemma}{Lemma}[section]

\newtheorem{proposition}{Proposition}[section]

\numberwithin{equation}{section}

\usepackage{graphicx}

\title{Normalized solutions of quasilinear Schr\"odinger-Poisson system with critical nonlinear term in bounded domain}

\author{Li Chen,\enspace
  Li Wang\thanks{Corresponding author. chenlimath0318@163.com (L. Chen), wangli.423@163.com (L. Wang).}\\{\small
	College of science, East China Jiaotong University, Nanchang, 330013, China}}
\date{}

\begin{document}

\maketitle

\begin{abstract}

This work examines a quasilinear Schr\"odinger-Poisson system involving a critical nonlinearity, expressed as
\[\begin{cases}
-\Delta u + \phi u + \lambda u = |u|^{q-2} u + |u|^4 u, & x \in \Omega_r, \\
-\Delta \phi - \varepsilon^4 \Delta_4 \phi = u^2, & x \in \Omega_r, \\
u = \phi = 0, & x \in \partial \Omega_r
\end{cases}\]
subject to the normalized condition
\[
\int_{\Omega_r} |u|^2 \d x = b^2.
\]
Here $\varepsilon > 0$, $q \in (2, 8/3)$, $\Omega_r \subset \mathbb R^3$ is a bounded domain.
By means of a truncation method combined with genus theory, we establish the existence of multiple families of normalized solutions. Due to the presence of a critical exponent in the nonlinear term, the associated energy functional fails to satisfy the usual compactness properties. To address this issue, we invoke the concentration-compactness principle. Furthermore, we derive the asymptotic result that the aforementioned system reduces to the classical Schr\"odinger-Poisson system (with $\varepsilon = 0$).
Our findings extend several recent results concerning problems of this type.
\end{abstract}

\begingroup

\endgroup

\section{Introduction and main results}

This paper focuses on the analysis of the following coupled system:
\begin{equation}\label{eq:1.1}
  \begin{cases}
    -\Delta u + \phi u + \lambda u = |u|^{q-2}u + |u|^4u, & x \in \Omega_r, \\
    -\Delta \phi - \varepsilon^4 \Delta_4 \phi = u^2, & x \in \Omega_r, \\
    u = \phi = 0, & x \in \partial\Omega_r
  \end{cases}
\end{equation}
subject to the mass constraint
\begin{equation}\label{eq:1.2}
  \int_{\Omega_r} u^2 \d x = b^2.
\end{equation}
In the above, $\varepsilon$ is a positive constant, and $\Delta_4 \phi$ represents the 4-Laplacian operator defined by $\Delta_4 \phi = \text{div}(|\nabla \phi|^2 \nabla \phi)$. The quantity $\lambda$ serves as an unknown frequency parameter, emerging as a part of the solution. Equations of the form~\eqref{eq:1.1} arise in the modeling of quantum phenomena within semiconductor nanostructures of ultra-small scale, where interactions between quantum confinement and longitudinal electric field oscillations during beam propagation are considered.
For more corresponding physical background, please refer to~\cite{tmp1,tmp21}.
To the best of our knowledge, the first rigorous mathematical treatment of systems sharing structural similarities with~\eqref{eq:1.1} was presented in the works of~\cite{Illner1998} and~\cite{Illner1997}. Later on, this category of coupled equations was designated as the quasilinear Schr\"odinger-Poisson system, representing a generalized framework that extends the scope of the traditional Schr\"odinger-Poisson system.

Illner, Kavian and Lange in~\cite{Illner1998} investigated a quasi-linear Schr\"odinger-Poisson system within unit cube $Q = [0,1]^v$ ($v = 1$, $2$, $3$) to model quantum structures in semiconductors. The evolution of the system is governed by the following coupled equations:
\begin{equation}\label{eq:picirc1}
  \begin{cases}
    i \partial_t \psi_m = -\frac12 \Delta \psi_m + (V + \tilde V) \psi_m, & (x,t) \in Q \times [0, +\infty),\\
    -\nabla \cdot ((\varepsilon_0 + \varepsilon_1 |\nabla V|^2)\nabla V) = n - n^*, & (x,t) \in Q \times [0, +\infty),
  \end{cases}
\end{equation}
where the charge density is defined by $n(x, t) = \sum_{m=1}^\infty \lambda_m |\psi_m|^2$, $n^*$ and $\tilde V$ are given time-independent $1$-periodic functions describing a dopant density and effective exterior potential.
Regarding the methodology, by seeking stationary solutions of the form $\psi_m(x, t) = e^{-i\omega_m t}\varphi_m(x)$, the authors reduced the system~\eqref{eq:picirc1} to a nonlinear eigenvalue problem. Through an operator-theoretic analysis of the quasi-linear Poisson equation and by treating the potential $V$ as a functional of $\varphi_m$, they constructed the energy functional in appropriate Sobolev spaces. By applying genus theory within a variational framework and rigorously verifying that the functional satisfies the Palais-Smale (PS) compactness condition, they proved the existence of infinitely many distinct stationary states on a unit cube, with the corresponding eigenfrequencies satisfying $\omega_m \to \infty$.
Later,
Benmilh and Kavian in~\cite{Benmlih2008} investigated the quasilinear Schr\"odinger-Poisson system defined on $\mathbb{R}^3$:
\begin{equation}\label{eq:1.3}
  \begin{cases}
    -\frac{1}{2}\Delta u + (V + \tilde V)u + \omega u = 0, & x \in \mathbb R^3,\\
    -\operatorname{div}\bigl[(1 + \varepsilon^4 |\nabla V|^2)\nabla V\bigr] = |u|^2 - n^*, & x \in \mathbb R^3,
  \end{cases}
\end{equation}
where $\varepsilon > 0$ is a nonlinearity parameter. Under suitable assumptions on the potential $\tilde V$ and the doping density $n^*$, by minimizing the corresponding energy functional $E_\varepsilon$, the authors proved the existence of a nontrivial ground state solution $(u_\varepsilon, V_\varepsilon)$ for all sufficiently small $\varepsilon > 0$. Furthermore, they analyzed the asymptotic behavior of the solutions as $\varepsilon \to 0^+$ and showed that $(u_\varepsilon, V_\varepsilon)$ converges to $(u_0, V_0)$, the ground state solution of the corresponding system with $\varepsilon = 0$ (i.e., with a linear Poisson equation). The main innovations of this work lie in handling the difficulties raised by the nonlinear term $|\nabla V|^2 \nabla V$ in the Poisson equation on an unbounded domain, and in establishing the convergence of solutions from the nonlinear system to the linear one.
When periodic boundary conditions are enforced, the system~\eqref{eq:1.3} in bounded domain was further explored in~\cite{Illner1998}, where the Krasnoselskii genus method was used to derive infinitely many $L^2$-normalized solutions.

Figueiredo and Siciliano in~\cite{Figueiredo2020} examined the following quasilinear Schr\"odinger-Poisson system featuring critical growth nonlinearity,
\begin{equation}\label{eq:1.4}
\begin{cases}
-\Delta u + u + \phi u = \lambda f(x,u) + |u|^4u, & x \in \mathbb{R}^3,\\
-\Delta \phi - \varepsilon^4 \Delta_4 \phi = u^2, & x \in \mathbb{R}^3.
\end{cases}
\end{equation}
Here, the nonlinear term $f$ meets subcritical growth and the Ambrosetti-Rabinowitz (AR) condition. To overcome the difficulties arising from the lack of favorable properties of the nonlocal term and the loss of compactness in the whole space, a truncation technique is introduced to construct a modified functional, which restores compactness within a finite energy range. Subsequently, the geometric structure of the mountain pass theorem is verified for the truncated functional, yielding a bounded (PS) sequence. In order to extract a strongly convergent subsequence from this sequence, Lions' concentration compactness principle is applied to analyze possible energy concentration phenomena, thereby establishing the existence of solutions.
The authors demonstrated the existence of a threshold value $\lambda^* > 0$ such that for all $\lambda \geq \lambda^*$ and $\varepsilon > 0$, the system has a nonnegative solution $(u_{\lambda,\varepsilon}, \phi_{\lambda,\varepsilon})$. Additionally, as $\lambda \to +\infty$, the norms of solution approach zero, and as $\varepsilon \to 0^+$, the solution family converges to a solution of the classical Schr\"odinger-Poisson system (when $\varepsilon = 0$). The key contributions of this work include addressing the lack of compactness and the absence of an explicit solution formula introduced by the quasilinear term, achieved by constructing a truncated functional and combining variational techniques with the concentration-compactness principle. This methodology not only establishes existence and asymptotic results for the critical-growth case in the entire space but also applies to supercritical nonlinearities.

For quasilinear Schr\"odinger-Poisson systems that violate the (AR) condition (and thus differ from system~\eqref{eq:1.4}), constructing bounded (PS) sequences cannot be achieved through conventional variational methods. Relevant discussions on this challenge are presented in~\cite{Wei2022} and~\cite{Wei2025}.

Quasilinear Schr\"odinger-Poisson systems defined on bounded domains have also drawn significant attention; in~\cite{Figueiredo2019}, Figueiredo and Siciliano analyzed the following system on a bounded two-dimensional domain:
\[
\begin{cases}
  - \Delta u + \phi u = f(u), & x \in \Omega, \\
  - \Delta \phi - \varepsilon^4 \Delta_4 \phi = u^2, & x \in \Omega, \\
  u = \phi = 0, & x \in \partial \Omega,
\end{cases}
\]
where $f$ satisfies exponential critical growth criteria. Under suitable assumptions, the authors used variational methods and the mountain pass theorem, introducing a pivotal truncated functional $J_\varepsilon^T$ to address challenges such as the lack of compactness induced by the quasilinear term and the strong indefiniteness of the energy functional, they first showed that for every $\varepsilon > 0$, the system has at least one nonnegative solution ($u_\varepsilon$, $\phi_\varepsilon$). Additionally, as $\varepsilon \to 0^+$, this solution sequence converges to a nontrivial solution of the corresponding classical Schr\"odinger-Poisson system (with $\varepsilon = 0$). 

Recently, Peng, Jia and Huang in~\cite{peng2022quasilinear} investigated a class of quasilinear Schr\"odinger-Poisson systems with exponential critical growth and logarithmic nonlinearity:
\[
\begin{cases}
  - \Delta u + \phi u = |u|^{p-2}u \log |u|^2 + \lambda f(x,u), & x \in \Omega, \\
  - \Delta \phi - \varepsilon^4 \Delta_4 \phi = u^2, & x \in \Omega, \\
  u = \phi = 0, & x \in \partial \Omega
\end{cases}
\]
where $\Omega \subset \mathbb{R}^2$ is a bounded domain. $f$ satisfies some proper technique conditions. Under suitable assumptions, the authors employed variational methods, a reduction argument, and a truncation technique to overcome difficulties arising from the loss of homogeneity due to the quasilinear term, the sign-changing logarithmic term, and the lack of compactness caused by the exponential critical growth. They proved that for sufficiently large $\lambda$, the system admits at least one pair of nonnegative solutions $(u_{\varepsilon,\lambda}, \phi_{\varepsilon,\lambda})$, and further analyzed the asymptotic behavior of the solutions with respect to the parameters $\lambda \to +\infty$ and $\varepsilon \to 0^+$.

Regarding normalized solutions of system \eqref{eq:1.1}, the literature remains sparse. In our study, two principal challenges arise: first, the energy functional corresponding to \eqref{eq:1.1}-\eqref{eq:1.2} does not possess an extremal geometric structure, which prevents the direct application of minimization methods. To overcome this, we restrict the exponent $q$ of the nonlinear term to the range $q\in(2,8/3)$; this choice allows us to bypass the aforementioned obstacle. Second, the nonlocal terms in the system introduce a compactness issue: because they involve Sobolev-critical exponents, standard compactness techniques fail. To restore compactness, inspired by~\cite{Gao2024}, we employ the concentration-compactness principle to analyze the behavior of (PS)-sequences for the functional.

Building on the aforementioned works, we now present the main result of this article concerning problems \eqref{eq:1.1}-\eqref{eq:1.2}.

\begin{theorem}\label{thm:1.1}
For any fixed $k \in \mathbb N$, one can find a constant $b^* > 0$ such that whenever
$b \in (0, b^*)$, equation~\eqref{eq:1.1} posses at least $k$ distinct solutions
\[
  (u_j, \phi_j, \lambda_j) \in H_0^1(\Omega_r) \times (H_0^1(\Omega_r) \cap W_0^{1,4}(\Omega_r)) \times \mathbb R
\]
satisfying the normalization condition $\int_{\Omega_r} u_j^2 \d x = b^2$,
$j = 1$, $2$, $\ldots\,$,~$k$.
\end{theorem}

\begin{theorem}\label{thm:1.2}
  For any $\varepsilon > 0$, denote $(u_{j,\varepsilon}, \phi_{j,\varepsilon}, \lambda_{j,\varepsilon}) \subset S_{r,b} \times (H_0^1(\Omega_r) \cap W_0^{1,4}(\Omega_r)) \times \mathbb R$, $j = 1$, $2$, $3$, $\ldots\,$,~$k$ are the solutions of~\eqref{eq:1.1} obtained from Theorem~\ref{thm:1.1}.
  Under the assumption of Theorem~\ref{thm:1.1}, when $\varepsilon \to 0$, there exists $(u_{j,0}, \phi_{j,0}, \lambda_{j,0}) \subset S_{r,b} \times H_0^1(\Omega_r) \times \mathbb R$, which ensures
  \[
    \lim_{\varepsilon\to0} u_{j,\varepsilon} = u_{j,0} \quad \text{in } H_0^1(\Omega_r), \quad
    \lim_{\varepsilon\to0} \phi_{j,\varepsilon}(u_{j,\varepsilon}) = \phi_{j,0}(u_{j,0}) \quad \text{in } H_0^1(\Omega_r), \quad j = 1,\ 2,\ \ldots\,,~k,
  \]
  where $(u_{j,0}, \phi_{j,0})$ are the solutions of the following Schr\"odinger-Poisson
  \begin{equation}
    \begin{cases}
      -\Delta u + \phi u = \lambda u + |u|^{q-2} u + |u|^4 u, & x \in \Omega_r,\\
      -\Delta \phi = u^2, & x \in  \Omega_r,\\
      u = \phi = 0, & x \in \partial\Omega_r,\\
      \int_{\Omega_r} u^2 \d x = b^2
    \end{cases}
  \end{equation}
  with $\lambda = \lambda_{j,0}$.
\end{theorem}

\section{Preliminaries}


Throughout this work, we fix a bounded smooth domain $\Omega \subset \mathbb R^3$.
For any scaling parameter $r > 0$, we set $\Omega_r := \{r x \in \mathbb R^3: x \in\Omega\}$.
The standard Sobolev spaces are denoted by $H_0^1(\Omega_r)$ and $W_0^{1,4}(\Omega_r)$.
In particular, the inner product and norm on $H_0^1(\Omega_r)$ are given by
\[
(u,v) := \int_{\Omega_r} \nabla u \cdot \nabla v \d x, \qquad
\|u\| := \bigg( \int_{\Omega_r} |\nabla u|^2 \d x \bigg)^{1/2}.
\]

The $L^2(\Omega_r)$ inner product is
\[
(u,v)_2 := \int_{\Omega_r} u v \d x.
\]
For any exponent $1 \leq p < 6$, we define the $L^p(\Omega_r)$ norm as
\[
\|u\|_p := \bigg( \int_{\Omega_r} |u|^p \d x \bigg)^{1/p},
\]
while the norm in $L^\infty(\Omega_r)$ is denoted by $|u|_\infty$.

We introduce the function space
\[
\mathbb{H}:=H_0^1(\Omega_r)\cap W_0^{1,4}(\Omega_r),
\]
equipped with the norm
\[
	\|\phi\|_{\mathbb{H}}
	:= \|\nabla \phi\|_2 + \|\nabla \phi\|_4 = \biggl(\int_{\Omega_r} |\nabla \phi|^2 \d x\biggr)^{\frac12} + \biggl(\int_{\Omega_r} |\nabla \phi|^4 \d x\biggr)^{\frac14}.
\]

Our analysis begin with the quasilinear Poisson equation~\eqref{eq:1.1}.
According to the results established in~\cite{Figueiredo2019}, for any $u$, the mapping
\[
	u^2 : \phi \in \mathbb H \longmapsto \int_{\Omega_r} \phi u^2 \d x \in \mathbb R
\]
is both linear and continuous, hence $u^2 \in \mathbb H$.
Consequently, the boundary value problem
\begin{equation}\label{eq:2.1}
	\begin{cases}
		-\Delta \phi - \varepsilon^4 \Delta_4 \phi = u^2, & x \in \Omega_r,\\
		\phi = 0, & x \in \partial\Omega_r
	\end{cases}
\end{equation}
possesses a unique solution $\phi_\varepsilon$. This solution is, in fact, the unique critical
point of the functional
\begin{equation}\label{eq:2.2}
	\phi_\varepsilon \in \mathbb H \longmapsto \frac12 \int_{\Omega_r} |\nabla \phi_\varepsilon|^2 \d x
	+ \frac{\varepsilon^4}{4} \int_{\Omega_r} |\nabla \phi_\varepsilon|^4 \d x
	- \int_{\Omega_r} \phi_\varepsilon u^2 \d x.
\end{equation}
That is, for any $\varphi \in \mathbb H$ we have
\begin{equation}\label{eq:2.3*}
	\int_{\Omega_r} (\nabla \phi_\varepsilon(u) \nabla \varphi + \varepsilon^4 |\nabla \phi_\varepsilon(u)|^2 \nabla \phi_\varepsilon(u)\nabla \varphi) \d x
= \int_{\Omega_r} u^2 \varphi \d x.
\end{equation}
For any $u \in H_0^1(\Omega_r)$,
using~\eqref{eq:2.3*}, H\"older's inequality and the Sobolev inequality yield
\begin{equation}
	\|\nabla\phi_\varepsilon(u)\|_2^2 + \varepsilon^4 \|\nabla \phi_\varepsilon(u)\|_4^4
= \int_{\Omega_r} \phi_\varepsilon(u) u^2 \d x \leq
\|\phi_\varepsilon(u)\|_6 \|u^2\|_{\frac65} \leq C \|\nabla \phi_\varepsilon(u)\|_2 \|u\|_{\frac{12}5}^2,
\end{equation}
implying
\begin{equation}
	\|\nabla \phi_\varepsilon(u)\|_2 \leq C\|u\|_{\frac{12}{5}}^2.
\end{equation}
Hence we obtain
\begin{equation}\label{eq:2.5*}
	\|\nabla \phi_\varepsilon(u)\|_2^2 + \varepsilon^4\|\nabla \phi_\varepsilon(u)\|_4^4 \leq C\|u\|_{\frac{12}{5}}^4 \leq C\|u\|^4.
\end{equation}

Continuity of $\phi_\varepsilon$ follows from the next lemma, whose proof is identical to that in~\cite{Figueiredo2019}.

\begin{lemma}
	Let $g_n \to g$ in $\mathbb H'$. Then
	\[
		\int_{\Omega_r} |\nabla \phi_\varepsilon(g_n)| \d x \to
		\int_{\Omega_r} |\nabla \phi_\varepsilon(g)|^2 \d x,\quad
		\int_{\Omega_r} |\nabla \phi_\varepsilon(g_n)|^4 \d x \to
		\int_{\Omega_r} |\nabla \phi_\varepsilon(g)|^4 \d x.
	\]
	In particular, the operator $\phi_\varepsilon$ is continuous and
	\[
		\phi_\varepsilon(g_n) \to \phi_\varepsilon(g) \quad
		\text{in } L^\infty(\Omega_r).
	\]
\end{lemma}
Throughout the rest of the paper, $\phi_\varepsilon = \phi_\varepsilon(u)$ will denote the unique solution of~\eqref{eq:2.1} for a given $u$, which satisfies
\[
	\int_{\Omega_r} |\nabla \phi_\varepsilon(u)|^2 \d x + \varepsilon^4 \int_{\Omega_r} |\nabla \phi_\varepsilon(u)|^4 \d x = \int_{\Omega_r} \phi_\varepsilon(u) u^2\d x.
\]
The following lemma summarize further properties of $\phi_\varepsilon(u)$, proved in~\cite{Figueiredo2019}.

\begin{lemma}\label{lem:2.2*}
	Let $q \in [1, +\infty)$.
	If $\{u_n\}$ converges to $\tilde\omega$ in $L^q(\Omega_r)$, then
  \begin{enumerate}[label = \textup{(\alph*)}, leftmargin = 2em, nosep]
    \item $\displaystyle\lim_{n\to\infty} \int_{\Omega_r} \phi_\varepsilon(u_n) u_n^2 \d x
	= \int_{\Omega_r} \phi_\varepsilon(\tilde\omega) \tilde\omega^2 \d x$,
    \item $\phi_\varepsilon(u_n) \to \phi_\varepsilon(\tilde \omega)$ in
    $L^\infty(\Omega_r)$,
    \item $\displaystyle \lim_{n\to\infty} \int_{\Omega_r} \phi_\varepsilon(u_n) u_n v\d x
	= \int_{\Omega_r} \phi_\varepsilon(\tilde\omega) \tilde\omega v \d x$, for all $v \in H_0^1(\Omega_r)$.
  \end{enumerate}
\end{lemma}

\begin{lemma}[\cite{MR4812883} Gagliardo-Nirenberg inequality]\label{lem:2.3}
For any $N\geqslant 2$ and $s\in(2,2^{*})$, there exists a constant $C_{N,s}>0$ depending on $N$ and $s$ such that
\[
\int_{\mathbb{R}^{N}}|u|^s\d x\leqslant C_{N,s}\bigg(\int_{\mathbb{R}^{N}}|u|^2\d x\bigg)^{\frac{2s-N(s-2)}{4}}\bigg(\int_{\mathbb{R}^{N}}|\nabla u|^2\d x\bigg)^{\frac{N(s-2)}{4}},\quad \forall u \in H^1(\mathbb{R}^{N}).
\]
\end{lemma}

In order to prove Theorem \ref{thm:1.1}, we employ the following concentration-compactness principle from \cite{Peng2024}.

\begin{lemma}\label{lem:3.1}
\[
|\nabla u_n|^2\rightharpoonup \mu \geqslant|\nabla u|_2^2+\sum_{j\in\mathcal{I}}\mu_j \delta_{x_j},\quad|u_n|^6\rightharpoonup\nu=|u|^6+\sum_{j\in\mathcal{I}}\nu_j\delta_{x_j},\quad\sum_{j\in\mathcal{I}}\nu_j^{\frac{1}{3}}<+\infty,
\]
where $\mu$, $\nu$, $\mu_j$ and $\nu_j$ are positive measures, $\mathcal{I}$ is an at most countable index set, $\{x_j\}\subset\mathbb{R}^3$ and $\delta_{x_j}$ is the Dirac mass at $x_j$. Moreover,
\[
S\nu_j^{\frac{1}{3}}\leqslant\mu_j ,\quad\forall j\in\mathcal{I},
\]
where $S$ is the optimal Sobolev constant defined by
\[
S=\inf\Biggl\{\int_{\mathbb{R}^3}|u|^2\d x:\ \int_{\mathbb{R}^3}|u|^{2^{*}}\d x=1\Biggr\}.
\]
\end{lemma}

\begin{lemma}[\cite{Chabrowski1995}]\label{lem:3.2}
Let $\{u_n\}$ be a bounded sequence in Lemma \ref{lem:3.1} and define
\begin{align*}
 \mu_{\infty}&:=\lim_{R\to\infty}\limsup_{n\to\infty}\int_{|x|>R}|\nabla u_n|^2\d x,\\
\nu_{\infty}&:=\lim_{R\to\infty}\limsup_{n\to\infty}\int_{|x|>R}|u_n|^6\d x.
\end{align*}
Then
\begin{align*}
\limsup_{n\to\infty}|\nabla u_n|^2_2&=\int_{\mathbb{R}^3}\d \mu + \mu_{\infty},\\
\limsup_{n\to\infty}|u_n|^6_6&=\int_{\mathbb R^3}\d \nu+\nu_{\infty}.
\end{align*}
Furthermore,
\[
S\nu^{\frac13}_{\infty}\leqslant \mu_{\infty}.
\]
\end{lemma}

To begin with, we revisit the concept of genus. Let $X$ denote a Banach space and let ${\mathfrak B}$ be a subset of $X$.
The set ${\mathfrak B}$ is termed symmetric if $u \in {\mathfrak B}$ implies $-u \in {\mathfrak B}$.

We introduce the collection $\Sigma$ as follows:
\[
\Sigma:=\{\mathfrak B\subset X\backslash\{0\}:\text{ $\mathfrak B$ is closed and symmetric with respect to the origin}\}.
\]

For any $\mathfrak B \in \Sigma$, the genus $\sigma(\mathfrak B)$ is defined by
\[
\sigma(\mathfrak B):=\begin{cases}
0,&\text{if }\mathfrak B=\emptyset,\\
\inf\{k\in\mathbb{N}:\text{ there exists an odd }g\in C(\mathfrak B,\mathbb{R}^k\backslash\{0\})\},\\
+\infty,&\text{if no such odd map exists},
\end{cases}
\]
we further define $\Sigma_k := \{A \in \Sigma | \sigma(A) \geq k\}$.

We now define a mainfold $S_{r,b}$ by
\[
  S_{r,b} := \{u \in H_0^1(\Omega_r) | (u, u)_{L^2} = b^2\},
\]
where the topology on $S_{r,b}$ is inherited from $H_0^1(\Omega_r)$. At any point $v \in S_{r,b}$, the tangent space is given by
\[
  T_{u} S_{r,b} = \{v \in H_0^1 (\Omega_r): (u, v)_{L^2} = 0\}.
\]
For $E \in C^1(H_0^1(\Omega_r), \mathbb R)$,
the restriction $E\big|_{S_{r,b}}$ belongs to $C^1$.
The norm of its derivative at $u \in S_{r,b}$ is
\[
  \|E\big|_{S_{r,b}}(u)\| = \sup_{\|v\|\leq 1, v \in T_{u} S_{r,b}}
  |\langle E'(u), v\rangle|.
\]
Note that $S_{r,b}$ is symmetric with respect to $0$ and $0 \notin S_{r,b}$.
Let $\Sigma(S_{r,b})$ denote the family of closed symmetric subsets of $S_{r,b}$,
and for each $j \in \mathbb N$, set $\Gamma_j := \{A \in \Sigma(S_{r,b})|\sigma(A) \geq j\}$.

\begin{proposition}[\cite{Jeanjean2018}]\label{prop:2.2}
Let $E\in C^1(H_0^1(\Omega_r),\mathbb{R})$ be an even functional. Assume that $\Gamma_j\neq\emptyset$ for each $j\in\mathbb{N}$, $E|_{S_{r,b}}$ is bounded from below and satisfies the $(PS)_d$ condition for all $d<0$. Define
\[
d_j=\inf_{A\in\Gamma_j}\sup_{u \in A}E(u),\quad j=1,2,\ldots,n.
\]
Then:
\begin{enumerate}[label=(\roman*)]
    \item $-\infty<d_1\leqslant d_2\leqslant\cdots\leqslant d_n$ and each $d_j(j=1,2,\ldots,n)$ with $d_j<0$ is a critical value of $E|_{S_{r,b}}$.
    \item If $d:=d_j=d_{j+1}=\cdots=d_{j+l-1}<0$ for some $j$, $l\geqslant 1$, then $\sigma(K_d)\geqslant l$, where $K_d$ denotes the set of critical points of $E|_{S_{r,b}}$ at level $d$. In particular, $E|_{S_{r,b}}$ admits at least $l$ critical points at level $d$.
\end{enumerate}
\end{proposition}

\section{Proof of Theorem \ref{thm:1.1}}

In this section we study solutions of equation (\ref{eq:1.1}) with a prescribed $L^2$-norm. Given $b>0$, we work on the constraint manifold
\[
S_{r,b} = \biggl\{u \in H_0^1(\Omega_r):\ \int_{\Omega_r}u^2\d x=b^2\biggr\}.
\]
We consider system (\ref{eq:1.1}) in the space $H_0^1(\Omega_r)\times \mathbb H$. For any $u \in S_{r,b}$, define the energy functional
\begin{align*}
F(u,\phi_\varepsilon) &:=\frac{1}{2}\int_{\Omega_r}|\nabla u|^2\d x+\frac{1}{2}\int_{\Omega_r}\phi_\varepsilon u^2\d x-\frac{1}{4}\int_{\Omega_r}|\nabla\phi_\varepsilon|^2\d x \\
          &\quad-\frac{\varepsilon^4}{8}\int_{\Omega_r}|\nabla\phi_\varepsilon|^4\d x-\frac{1}{q}\int_{\Omega_r}|u|^q\d x-\frac{1}{6}\int_{\Omega_r}|u|^6\d x
\end{align*}
on $H_0^1(\Omega_r)\times \mathbb H$; its critical points correspond to solutions of system (\ref{eq:1.1})-(\ref{eq:1.2}). Following the reduction method introduced in \cite{Benci1998} and \cite{Benci2002}, we reduce it to a single-variable functional. For $u \in S_{r,b}$, define
\begin{align*}
E(u):=F(u,\phi_\varepsilon(u)) &=\frac{1}{2}\int_{\Omega_r}|\nabla u|^2\d x+\frac{1}{4}\int_{\Omega_r}|\nabla\phi_\varepsilon(u)|^2\d x+\frac{3\varepsilon^4}{8}\int_{\Omega_r}|\nabla\phi_\varepsilon(u)|^4\d x \\
                                   &\quad-\frac1q\int_{\Omega_r}|u|^q\d x-\frac{1}{6}\int_{\Omega_r}|u|^6\d x.
\end{align*}

For any $u \in S_{r,b}$, the Gagliardo-Nirenberg-Sobolev inequality and the definition of the best Sobolev constant $S$ yield
\begin{align*}
    E(u)&=\frac{1}{2}\int_{\Omega_r}|\nabla u|^2\d x+\frac{1}{4}\int_{\Omega_r}|\nabla\phi_\varepsilon(u)|^2\d x+\frac{3\varepsilon^4}{8}\int_{\Omega_r}|\nabla\phi_\varepsilon(u)|^4\d x \\
     &\quad-\frac1q\int_{\Omega_r}|u|^q\d x-\frac{1}{6}\int_{\Omega_r}|u|^6\d x \\
    &\geqslant\frac{1}{2}\int_{\Omega_r}|\nabla u|^2\d x-\frac1q\int_{\Omega_r}|u|^q\d x-\frac{1}{6}\int_{\Omega_r}|u|^6\d x \\
    &\geqslant\frac{1}{2}\|\nabla u\|^2_2-\frac1qb^{3-\frac q2}C_{3,q}\|\nabla u\|_2^{\frac{3(q-2)}2}-\frac{1}{6S^3}\|\nabla u\|^6_2 \\
    &:=\mathcal G(\|\nabla u\|_2),
\end{align*}
where
\[
\mathcal G(t) :=\frac{1}{2}t^2-\frac1qb^{3-\frac q2}C_{3,q}t^{\frac{3(q-2)}2}-\frac{1}{6S^3}t^6.
\]
Given that $q \in (2,8/3)$, we can find a constant $b^* > 0$ such that whenever $b \in (0, b^*)$, the function $\mathcal G(t)$ admits a local positive maximum.
Furthermore, there exist positive constants $R_1$ and $R_2$ satisfying $0 < R_1 < R_2 < +\infty$ such that $\mathcal G(t) > 0$ on the interval $(R_1, R_2)$,
whereas $\mathcal G(t) < 0$ when $t \in (0, R_1)$ or $t \in (R_2, +\infty)$.
We now introduce a smooth cutoff function $T \in C^\infty (\mathbb R^+, [0,1])$, such that
\[
T(t):=\begin{cases}
1,&\text{if }t\leqslant R_1,\\
0,&\text{if }t\geqslant R_2.
\end{cases}
\]
Define the truncated functional
\begin{align*}
E^T(u)&:=\frac{1}{2}\int_{\Omega_r}|\nabla u|^2\d x+\frac{1}{4}\int_{\Omega_r}|\nabla\phi_\varepsilon(u)|^2\d x+\frac{3\varepsilon^4}{8}\int_{\Omega_r}|\nabla\phi_\varepsilon(u)|^4\d x \\
           &\quad-\frac1q\int_{\Omega_r}|u|^q\d x-\frac{T(\|\nabla u\|_2)}{6}\int_{\Omega_r}|u|^6\d x,
\end{align*}
which is of class $C^1$. By construction,
\[
E^T(u)\geqslant\frac{1}{2}\|\nabla u\|^2_2-\frac1qb^{\frac{6-q}{2}}C_{3,q}\|\nabla u\|_2^{\frac{3(q-2)}{2}}-\frac{T(\|\nabla u\|_2)}{6S^3}\|\nabla u\|^6_2.
\]
Define
\[
\mathcal G^T(t):=\frac{1}{2}t^2-\frac1qb^{\frac{6-q}{2}}C_{3,q}t^{\frac{3(q-2)}{2}}-\frac{T(t)}{6S^3}t^6.
\]
For $b \in (0, b^*)$, we have $\mathcal G^T(t)<0$ for all $t\in(0,R_1)$ and $\mathcal G^T(t)>0$ for $t\in(R_1,+\infty)$.

Without loss of generality, we can choose $R_1 > 0$ small enough, such that
\begin{equation}\label{eq:3.1}
\frac{1}{2}r_0^2-\frac{1}{6S^3}r_0^6\geqslant 0\quad\text{for any }r_0\in[0,R_1]\text{ and }R_1^2<S^{\frac32}.
\end{equation}

\begin{lemma}\label{lem:3.1*}
    The functional $E^T$ possesses the following properties:
    \begin{enumerate}[label = \textup{(\roman*)}]
      \item $E^T \in C^1(H_0^1(\Omega_r),\mathbb R)$;
      \item $E^T$ is coercive and bounded from below on $S_{r,b}$.
      Moreover, if $E^T(u) \leq 0$, then $\|\nabla u\|_2 \leq R_1$,
      and $E^T(u) = E(u)$;
      \item $E^T$ satisfies the $(PS)_d$ condition on $S_{r,b}$.
    \end{enumerate}
\end{lemma}

\begin{proof}

(i) follows by standard arguments.

For (ii), if $\|\nabla u\|_2 \to +\infty$,
we have
\[
    E^T(u) \geq \frac12 \|\nabla u\|_2^2 - \frac 1q b^{\frac{6-q}{2}} C_{3,q}
    \|\nabla u\|_2^{\frac{3(q-2)}{2}} - \frac{T(\|\nabla u\|_2)}{6S^3} \|\nabla u\|_2^6 \to +\infty,
\]
which shows coercivity.
From the definition of $\mathcal G^T(t)$,
we obtain $E^T(u) \geq \mathcal G^T(\|\nabla u\|_2) = \mathcal G(\|\nabla u\|_2)$
when $\|\nabla u\|_2 \in [0, R_1]$.
Thus $E^T$ is bounded from below on $S_{r,b}$.
Note that if $\|\nabla u\|_2 > R_1$, since $\mathcal G^T(t) > 0$ for $t \in (R_1, +\infty)$,
we can know that $E^T(u) > 0$. It is a contradiction.
So, we can know if $E^T(u) \leq 0$, then $\|\nabla u\|_2 \leq R_1$.

For (iii), set
\[
    d:= \inf_{u \in S_{r,b}} E^T(u).
\]
For $u \in S_{1,b}$, consider the scaling
\[
    u_l(x) := l^{\frac32} u(lx) \quad \text{for } l \in \mathbb R,
\]
which preserves the $L^2$-norm for taking $r = \frac1l$:
\[
    \|u_l(x)\|_2 = \int_{\Omega_r} |l^{\frac32} u(lx)|^2 \d x
= l^3 \int_{\Omega_{\frac1l}} |u(lx)|^2 \d x
= \int_\Omega |u(y)|^2 \d y = b^2.
\]
Using H\"older's inequality and the embeddings $H_0^1(\Omega_r) \hookrightarrow L^6(\Omega_r)
\hookrightarrow L^2(\Omega_r)$, we estimate
\begin{align*}
    \int_{\Omega_r} |\phi_\varepsilon(u_l)|^2 \d x
    & \leq \biggl(\int_{\Omega_r}|\phi_\varepsilon(u_l)|^6 \d x\biggr)^{\frac13}
    |\Omega_r|^{\frac23}
    = \biggl(\int_{\Omega_r} |\phi_\varepsilon(u_l)|^6\d x\biggr)^{\frac13}
    (r^3|\Omega|)^{\frac23}\\
    & = r^2 \biggl(\int_{\Omega_r} |\phi_\varepsilon(u_l)|^6\d x\biggr)^{\frac13} |\Omega|^{\frac23}
    \leq r^2 \biggl(S^{-3} \|\nabla \phi_\varepsilon(u_l)\|_2^6\biggr)^{\frac13} |\Omega|^{\frac23}\\
    & = r^2S^{-1} \|\nabla \phi_\varepsilon(u_l)\|_2^2 |\Omega|^{\frac23},
\end{align*}
where $S$ is the constant in the Sobolev embedding $D^{1,2}(\mathbb R^3) \hookrightarrow L^6(\mathbb R^3)$.

Combing this with
\[
    \int_{\Omega_r} |\nabla \phi_\varepsilon(u_l)|^2 \d x \leq \int_{\Omega_r} |\nabla \phi_\varepsilon(u_l)|^2 \d x + \varepsilon^4 \int_{\Omega_r} |\nabla \phi_\varepsilon(u_l)|^4 \d x
= \int_{\Omega_r} u_l^2 \phi_\varepsilon(u_l) \d x,
\]
we obtain
\begin{equation}\label{eq:3.3}
    \int_{\Omega_r} u_l^2 \phi_\varepsilon(u_l) \d x
    \leq \|u_l\|_4^2 (rS^{-\frac12} |\Omega|^{\frac13}) \|\nabla \phi_\varepsilon(u_l)\|_2
    \leq (rS^{-\frac12}|\Omega|^{\frac13}) \|u_l\|_4^2 \biggl(\int_{\Omega_r} u_l^2 \phi_\varepsilon(u_l) \d x\biggr)^{\frac12}.
\end{equation}
Note that
\begin{equation}\label{eq:3.4}
  \begin{aligned}
    \|u_l\|_4^2 & = \biggl(\int_{\Omega_r}|u_l|^4 \d x\biggr)^{\frac12}
    = \biggl(\int_{\Omega_{\frac1l}} |l^{\frac32} u(lx)|^4 \d x\biggr)^{\frac12}\\
    & = l^3 \biggl(\int_{\Omega_{\frac1l}}|u(lx)|^4 \d x\biggr)^{\frac12}
    = l^{\frac32} \biggl(\int_{\Omega_{\frac1l}} |u(x)|^4 \d x\biggr)^{\frac12},
  \end{aligned}
\end{equation}
then, from~\eqref{eq:3.3} and~\eqref{eq:3.4} we get
\[
  \int_{\Omega_r} u_l^2 \phi_\varepsilon (u_l) \d x =
  \int_{\Omega_{\frac1l}} u_l^2 \phi_\varepsilon(u_l) \d x \leq
  \biggl(S^{-\frac12}|\Omega|^{\frac13}\biggr) l^{\frac12} \biggl(\int_\Omega|u|^4 \d x\biggr)^{\frac12}
  \biggl(\int_{\Omega_r} u_l^2\phi_\varepsilon(u_l) \d x\biggr)^{\frac12}.
\]
Setting $\tilde c: = S^{-\frac12} |\Omega|^{\frac13}$,
it follows that
\[
  \int_{\Omega_{\frac1l}} u_l^2 \phi_\varepsilon(u_l) \d x \leq \tilde c^2 l\int_\Omega |u|^4 \d x.
\]
Since
\[
  \int_{\Omega_{\frac1l}} u_l^2 \phi_\varepsilon(u_l) \d x
= \int_{\Omega_{\frac1l}} |\nabla \phi_\varepsilon(u_l)|^2 \d x + \varepsilon^4 \int_{\Omega_{\frac1l}} |\nabla \phi_\varepsilon(u_l)|^4 \d x,
\]
we have
\begin{equation}\label{eq:3.5}
  \int_{\Omega_{\frac1l}} |\nabla \phi_\varepsilon(u_l)|^2 \d x + \varepsilon^4 \int_{\Omega_{\frac1l}} |\nabla\phi_\varepsilon(u_l)|^4 \d x
  \leq \tilde c^2 l\int_{\Omega} |u|^4\d x.
\end{equation}
Now, substitute $u_l \in S_{r,b}$ into $E^T$, \eqref{eq:3.5} gives
\begin{align*}
  E^T(u_l) &
= \frac12 \int_{\Omega_{\frac1l}} |\nabla u_l|^2 \d x
+ \frac14 \int_{\Omega_{\frac1l}} |\nabla \phi_\varepsilon(u_l)|^2 \d x\\
& \quad + \frac{3\varepsilon^4}{8} \int_{\Omega_{\frac1l}} |\nabla\phi_\varepsilon(u_l)|^4 \d x
- \frac1q \int_{\Omega_{\frac1l}} |u_l|^q \d x
- \frac16 \int_{\Omega_{\frac1l}} |u_l|^6\d x\\
& \leq \frac{l^2}{2} \int_{\Omega} |\nabla u|^2 \d x + \biggl(\frac14 + \frac38\biggr) \tilde c l \int_{\Omega} |u|^4 \d x - \frac1q l^{\frac{3q}{2}-3}
\int_\Omega |u|^q \d x - \frac{l^6}{6} \int_\Omega |u|^6 \d x.
\end{align*}
By Ekeland's variational principle, there exists a (PS) sequence $\{u_n\} \subset S_{r,b}$ for $E^T(u)$.
Since $q \in (2, 8/3)$, we can choose $l_0 > 0$ small enough so that $E^T(u_{l_0}(x)) < 0$.

Now fix $r = \frac1{l_0}$ (hence $\Omega_r = \Omega_{\frac1{l_0}}$).
Thus the size of $\Omega_r$ is determined.
For sufficiently large $n$, we have $E^T(u_n) < 0$,
so by (ii) we obtain $\|\nabla u_n\|_2 \leq R_1$ and $E(u_n) = E^T(u_n)$.
Consequently $\{u_n\}$ is a (PS)$_d$ sequence for $E$, i.e.,
\[
  E(u_n) \to d < 0, \quad \|E'|_{S_{r,b}}(u_n)\| \to 0, \quad \text{as } n \to \infty.
\]
For convenience we shall write $\Omega_r$ instead of $\Omega_{\frac1{l_0}}$ from now on.

The sequence $\{u_n\}$ is bounded in $H_0^1(\Omega_r)$,
hence, up to a subsequence (still denoted by $\{u_n\}$), there exists $u \in H_0^1(\Omega_r)$ such that
\begin{equation}\label{eq:3.5*}
  u_n \rightharpoonup u \quad \text{in } H_0^1(\Omega_r), \quad
  u_n \to u \quad \text{in } L^t(\Omega_r), \quad
  \text{for } 2 \leq t < 6.
\end{equation}
Since $q \in (2, 8/3)$, we get
\[
  \lim_{n\to\infty} \int_{\Omega_r} |u_n|^q \d x
= \int_{\Omega_r} |u|^q \d x.
\]
We now show that $u \neq 0$.
Assume by contradiction that $u = 0$. Then
\[
  \lim_{n\to\infty} \int_{\Omega_r} |u_n|^q \d x
= \int_{\Omega_r} |u|^q \d x = 0.
\]
Consequently, from~\eqref{eq:3.1}, we get that
\begin{align*}
  0 > d & = \lim_{n\to \infty} E^T(u_n)\\
  & = \lim_{n\to\infty} \bigg(\frac12 \int_{\Omega_r} |\nabla u_n|^2 \d x + \frac14 \int_{\Omega_r} |\nabla \phi_\varepsilon(u_n)|^2 \d x \\
  & \quad + \frac{3\varepsilon^4}{8} \int_{\Omega_r} |\nabla \phi_\varepsilon(u_n)|^4 \d x - \frac1  q \int_{\Omega_r} |u_n|^q \d x - \frac16 \int_{\Omega_r} |u_n|^6 \d x\bigg)\\
  & \geq \lim_{n\to\infty} \biggl(\frac12\|\nabla u_n\|_2^2 - \frac 1  q\int_{\Omega_r} |u_n|^q \d x - \frac1{6S^3} \|\nabla u_n\|_2^6\biggr)\\
  & \geq -\frac1  q \lim_{n\to\infty} \int_{\Omega_r} |u_n|^q \d x,
\end{align*}
a contradiction. Hence $u \neq 0$.

By Proposition 5.2 in~\cite{Willem2012}, for any $\varphi \in H_0^1(\Omega_r)$,
there exists $\{\lambda_n\}\subset \mathbb R$ such that
\[
  \begin{aligned}
    & \int_{\Omega_r} \nabla u_n \nabla \varphi \d x
    + \int_{\Omega_r} \phi_\varepsilon(u_n) u_n \varphi \d x
    -  \int_{\Omega_r} |u_n|^{q-2} u_n\varphi \d x\\
    & \quad - \int_{\Omega_r} |u_n|^4 u_n \varphi \d x
    - \lambda_n \int_{\Omega_r} u_n \varphi \d x = o_n(1).
  \end{aligned}
\]
Taking $\varphi = u_n$ gives
\begin{equation}\label{eq:3.7*}
  \int_{\Omega_r} |\nabla u_n|^2 \d x
  + \int_{\Omega_r} \phi_\varepsilon(u_n) u_n^2 \d x
  -  \int_{\Omega_r} |u_n|^q \d x
  - \int_{\Omega_r} |u_n|^6 \d x
  = \lambda_n \int_{\Omega_r} u_n^2 \d x = \lambda_n b^2.
\end{equation}
Moreover,
\begin{equation}
  \int_{\Omega_r} |u_n|^q \d x \leq \hat c \|\nabla u_n\|_2^q
\end{equation}
for some constant $\hat c > 0$. From~\eqref{eq:2.5*} we obtain
\begin{equation}\label{eq:3.9*}
  \int_{\Omega_r} \phi_\varepsilon(u_n) u_n^2 \d x
= \int_{\Omega_r} |\nabla \phi_\varepsilon(u_n)|^2 \d x
+ \varepsilon^4 \int_{\Omega_r} |\nabla \phi_\varepsilon(u_n)|^4 \d x \leq C\|u_n\|^4.
\end{equation}
Thanks to the continuous embedding $H_0^1(\Omega_r) \hookrightarrow L^6(\Omega_r)$, the integral term $\int_{\Omega_r} |u_n|^6 \d x$ is bounded. Combining equations~\eqref{eq:3.7*}-\eqref{eq:3.9*} with the boundedness of $\{u_n\}$ in $H_0^1(\Omega_r)$, we conclude that the sequence $\{\lambda_n\}$ is bounded. By passing to a subsequence, we may assume without loss of generality that $\lambda_n \to \lambda \in \mathbb{R}$.
So, $u$ is a weak solution of
\begin{equation}\label{eq:3.10*}
  -\Delta u + \phi_\varepsilon(u) u + \lambda u
= |u|^{q-2} u + |u|^4 u.
\end{equation}
Indeed, for any $\varphi \in H_0^1(\Omega_r)$, the weak convergence $u_n \rightharpoonup u$
in $H_0^1(\Omega_r)$ and the strong convergence $u_n \to u$ in $L^p(\Omega_r)$ for $p \in [2,6)$ imply
\begin{gather}
  \int_{\Omega_r} \nabla u_n \nabla \varphi \d x \to \int_{\Omega_r} \nabla u \nabla \varphi \d x \quad \text{as } n \to \infty,\label{eq:3.11}\\
   \int_{\Omega_r} |u_n|^{q-2} u_n \varphi \d x \to
   \int_{\Omega_r} |u|^{q-2} u \varphi \d x \quad
  \text{as } n \to \infty,\\
  \lambda_n \int_{\Omega_r} u_n \varphi \d x \to \lambda \int_{\Omega_r}
  u \varphi \d x \quad \text{as } n \to \infty,\\
  \int_{\Omega_r} |u_n|^4 u_n \varphi \d x \to \int_{\Omega_r} |u|^4 u \varphi \d x \quad \text{as } n \to \infty,
\end{gather}
and by Lemma~\ref{lem:2.2*} (c),
\begin{equation}\label{eq:3.15}
  \int_{\Omega_r} \phi_\varepsilon(u_n) u_n \varphi \d x \to
  \int_{\Omega_r} \phi_\varepsilon(u) u \varphi \d x.
\end{equation}
Thus~\eqref{eq:3.10*} follows from~\eqref{eq:3.11}-\eqref{eq:3.15}.

We next prove that
\[
  \int_{\Omega_r} |\nabla u_n|^2 \d x \to \int_{\Omega_r} |\nabla u| \d x,
\qquad
  \int_{\Omega_r} |u_n|^6 \d x \to \int_{\Omega_r} |u|^6 \d x.
\]
Extend the functions $u_n$ by zero outside $\Omega_r$ and regard them as elements of $H^1(\mathbb R^3)$.
Then we may apply the concentration-compactness principle.
For sufficiently large $n$, $\|\nabla u_n\| \leq R_1$;
by Lemmas~\ref{lem:3.1} and~\ref{lem:3.2}, there exist positive measures $\mu$, $\nu$,
vanishing outside $\Omega_r$, such that
\[
  |\nabla u_n|^2 \rightharpoonup \mu, \quad
  |u_n|^6 \rightharpoonup \nu \quad
  \text{as } n \to \infty.
\]

\noindent\textbf{Step 1.} We show that the index set $\mathcal{I}$ is empty. Assume by contradiction that $\mathcal{I}\neq\emptyset$. For a fixed $j\in\mathcal{I}$, let $\varphi\in C^{\infty}_0(\mathbb{R}^{N})$ be a cut-off function with $\varphi\in[0,1]$, $\varphi=1$ in $B_{\frac{1}{2}}(0)$, and $\varphi\equiv 0$ in $\mathbb{R}^{N}\backslash B_1(0)$. For $\rho>0$, set
\[
\varphi_{\rho}(x)=\varphi\biggl(\frac{x-x_j}{\rho}\biggr)=\begin{cases}
1,&\text{if }|x-x_j|\leqslant\frac{\rho}{2},\\
0,&\text{if }|x-x_j|\geqslant\rho.
\end{cases}
\]
Since $\{u_n\}$ is bounded in $H^1(\mathbb{R}^{N})$, the sequence $\{\varphi_{\rho}u_n\}$ is also bounded in $H^1(\mathbb{R}^{N})$. Consequently,
\begin{equation}\label{eq:3.16}
\begin{split}
o(1)=\langle E^{\prime}(u_n),u_n\varphi_{\rho}\rangle&=\int_{\mathbb{R}^3}|\nabla u_n|^2\varphi_{\rho}\d x+\int_{\mathbb{R}^3}u_n\nabla u_n\nabla\varphi_{\rho}\d x - \int_{\mathbb{R}^3}|u_n|^q\varphi_{\rho}\d x\\
&\quad+\int_{\mathbb{R}^3}\phi_\varepsilon(u_n)u_n^2\varphi_{\rho}\d x-\int_{\mathbb{R}^3}|u_n|^6\varphi_{\rho}\d x=o_n(1).
\end{split}
\end{equation}
Using the weak convergence $u_n \rightharpoonup u$ in $H_0^1(\Omega_r)$, the strong convergence $u_n \to u$ in $L^p(\Omega_r)$ for $p \in [2,6)$, and the properties of $\varphi$, we obtain
\begin{gather}
\lim_{\rho\to 0}\lim_{n\rightarrow\infty}\int_{\mathbb{R}^3}|u_n|^q\varphi_{\rho}\d x=\lim_{\rho\to 0}\int_{\mathbb{R}^3}|u|^q\varphi_{\rho}\d x=\lim_{\rho\to 0}\int_{|x-x_j|\leqslant\rho}|u|^q\varphi_{\rho}\d x=0,\label{eq:3.6}\\
\lim_{\rho\to 0}\lim_{n\rightarrow\infty}\int_{\mathbb{R}^3}\phi_\varepsilon(u_n)u_n^2 \varphi_\rho\d x=\lim_{\rho\to 0}\int_{\mathbb{R}^3}\phi_\varepsilon(u) u^2 \varphi_\rho\d x=0.\label{eq:3.7}
\end{gather}
Applying H\"older's inequality together with the absolute continuity of Lebesgue integration, we also have
\[
\lim_{\rho\to 0}\lim_{n\rightarrow\infty}\int_{\mathbb{R}^3}u_n\nabla u_n\nabla\varphi_{\rho}\d x=0.
\]
Then, using (\ref{eq:3.5}) and Lemma~\ref{lem:3.1},
\begin{align*}
\lim_{\rho\to 0}\biggl\{\int_{\mathbb{R}^3}|u|^6\varphi_{\rho}\d x+\sum_{j\in\mathcal{I}}\nu_j\delta_{x_j}\varphi_{\rho}\d x\biggr\}
&=\lim_{\rho\to 0}\int_{\mathbb{R}^3}\varphi_{\rho}\d \nu=\lim_{\rho\to 0}\lim_{n\rightarrow\infty}\int_{\mathbb{R}^3}|u_n|^6\varphi_{\rho}\d x \\
&=\lim_{\rho\to 0}\lim_{n\rightarrow\infty}\int_{\mathbb{R}^3}|\nabla u_n|^2\varphi_{\rho}\d x=\lim_{\rho\to 0}\int_{\mathbb{R}^3}\varphi_{\rho}\d \mu\\
&\geqslant\lim_{\rho\to 0}\Biggl\{\int_{\mathbb{R}^3}|\nabla u|^2\varphi_{\rho}+\sum_{j\in\mathcal{I}}\mu_j \delta_{x_j}\varphi_{\rho}\d x\Biggr\}\geqslant\mu_j .
\end{align*}
By absolute continuity of the Lebesgue integral, $\nu_j\geqslant \mu_j $. Lemma \ref{lem:3.1} gives $\mu_j \geqslant S\nu_j^{\frac{1}{3}}$, hence $\mu_j \geqslant S\nu_j^{\frac{1}{3}}\geqslant S\mu_j ^{\frac{1}{3}}$. Therefore
\[
R_1^2\geqslant\lim_{\rho\to 0}\lim_{n\to\infty}\|\nabla u_n\|^2_2\geqslant\lim_{\rho\to 0}\lim_{n\to\infty}\int_{\mathbb{R}^3}|\nabla u_n|^2\varphi_{\rho}\d x=\lim_{\rho\to 0}\int_{\mathbb{R}^3}\varphi_{\rho}\d  m >\mu_j \geqslant S^{\frac32},
\]
contradicting (\ref{eq:3.1}). Thus $\mathcal{I}=\emptyset$, and we conclude $u_n\to u$ in $L^6_{\text{loc}}(\mathbb{R}^3)$, i.e.,
%
\[
  \lim_{n\to\infty} \int_{\Omega_r} |u_n|^6\d x = \int_{\Omega_r} |u|^6 \d x.
  \]
  Since $u_n \rightharpoonup u$, taking $\varphi = u$ as a test function in~\eqref{eq:3.10*}
  and subtracting~\eqref{eq:3.7*}, then
  \begin{align*}
    & \biggl(\int_{\Omega_r} |\nabla u_n|^2 \d x + \int_{\Omega_r} \phi_\varepsilon(u_n)u_n^2 \d x
    - \int_{\Omega_r} |u_n|^q \d x - \int_{\Omega_r} |u_n|^6 \d x
    - \lambda_n \int_{\Omega_r} u_n^2 \d x\biggr)\\
    & \quad - \biggl(\int_{\Omega_r} |\nabla u|^2 \d x
    + \int_{\Omega_r} \phi_\varepsilon(u)u \d x
    -  \int_{\Omega_r} |u|^q \d x - \int_{\Omega_r} |u|^6 \d x
    - \lambda \int_{\Omega_r} u^2 \d x\biggr)
    = o_n(1).
  \end{align*}
  Using $u_n \rightharpoonup u$ in $H_0^1(\Omega_r)$,
  $u_n \to u$ in $L^p(\Omega_r)$ for $p \in [2,6]$, together with Lemma~\ref{lem:2.2*} (a),
  we deduce
  \[
    \lim_{n\to\infty} \int_{\Omega_r} |\nabla u_n|^2 \d x = \int_{\Omega_r} |\nabla u|^2 \d x,
    \]
    i.e., $u_n \to u$ strongly in $H_0^1(\Omega_r)$.
\end{proof}

\begin{lemma}[\cite{Gao2024}]\label{lem:3.4}
Given $k\in\mathbb{N}$, there exist $\varepsilon_k:=\varepsilon(k)>0$ such that $\sigma((E^T)^{-\varepsilon})\geqslant k$ for every $0<\varepsilon\leqslant\varepsilon_k$.
\end{lemma}

Introduce the family
\[
\Sigma_k:=\{D\subset S_{r,b}:\ D\text{ is closed and symmetric, }\sigma(D)\geqslant k\},
\]
and define the minimax levels
\[
d_k:=\inf_{D\in\Sigma_k}\sup_{u \in D}E^T(u)>-\infty
\]
for all $k\in\mathbb{N}$. To prove Theorem \ref{thm:1.1}, we set
\[
K_d:=\{u \in S_{r,b}:(E^T)^{\prime}(u)=0,E^T(u)=d\}.
\]

\begin{lemma}\label{lem:3.5}
If $d=d_k=d_{k+1}=\cdots=d_{k+\ell}$, then $\sigma(K_d)\geqslant \ell+1$. In particular, $E^T(u)$ possesses at least $\ell+1$ nontrivial critical points.
\end{lemma}

\begin{proof}
For $\varepsilon > 0$, it is clear that $(E^T)^{-\varepsilon} \subset \Sigma$.
Given any fixed $k \in \mathbb N$, Lemma~\ref{lem:3.4} provides $\varepsilon_k > 0$ such that $(E^T)^{-\varepsilon} \in \Sigma_k$ for
$0 < \varepsilon \leq \varepsilon_k$.
Moreover,
\[
  d_k \leq \sup_{u \in (E^T)^{-\varepsilon_k}} E^T(u) = -\varepsilon_k < 0.
\]

If $0 > d = d_k = d_{k+1} = \cdots = d_{k+\ell}$, part (iii) of Lemma~\ref{lem:3.1*} implies that $E^T$ satisfies the (PS)$_d$ condition on $S_{r,b}$ for $d < 0$. Hence $K_d$ is compact.
The result now follows from the Theorem~2.1 in~\cite{Jeanjean2018}, which guarantees at least $\ell + 1$ non-trivial critical points for $E^T\big|_{S_{r,b}}$.
\end{proof}

\begin{proof}[\textup{\textbf{Proof of Theorem \ref{thm:1.1}}}]
By Lemma~\ref{lem:3.1*} (ii), the critical points of $E^T(u)$ obtained in Lemma \ref{lem:3.5} are precisely critical points of $\phi_\varepsilon$. This completes the proof of Theorem \ref{thm:1.1}.
\end{proof}

\section{Proof of Theorem~\ref{thm:1.2}}

\def \tmp {\phi_{j,\varepsilon}}

For convenient, in this section, for arbitrary $j = 1$, $2$, $3$, $\ldots\,$,~$k$, $(u_{j,\varepsilon}, \phi_{j,\varepsilon})$ are the solutions
of~\eqref{eq:2.1}.
From Lemma~\ref{lem:3.1*}, we can know that if $E^T(u) \leq 0$, then $\|\nabla u\|_2 < R_1$.
Since $E^T(u_{j,\varepsilon}) = d_{j,\varepsilon} < 0$, then we have
\[
  \|\nabla u_{j,\varepsilon}\|_2 < R_1, \quad
  j = 1,\ 2, \ \ldots\,,~k.
\]
Hence, $\{u_{j,\varepsilon}\}$ is bounded uniformly on $\varepsilon$ in $S_{r,b}$.
Additionally, under the meaning of the subsequence, there exists $u_{j,0} \in H_0^1(\Omega_r)$ such that
\begin{equation}\label{eq:circ1.2}
  u_{j,\varepsilon} \rightharpoonup u_{j,0} \quad \text{in } H_0^1(\Omega_r), \quad
  u_{j,\varepsilon} \to u_{j,0} \quad \text{in } L^p(\Omega_r) \quad
  \text{for } p \in [2,6).
\end{equation}
Similarly to the conclusion in Lemma~\ref{lem:3.1*} (iii), we can have
\[
  E^T(u_{j,\varepsilon}) = E(u_{j,\varepsilon}) \to d_{j,0} < 0,\quad
  \|E'\big|_{S_{r,b}} (u_{j,\varepsilon}) \| \to 0, \quad \text{as } \varepsilon \to 0.
\]
Then, combining with the Lagrange multipler theorem, $\forall \psi \in H_0^1(\Omega_r)$, there exists $\{\lambda_{j,\varepsilon}\} \subset \mathbb R$ such that
\begin{equation}\label{eq:4.1prime}
\begin{aligned}
\int_{\Omega_r} \nabla u_{j,\varepsilon} \nabla \psi \d x
& + \int_{\Omega_r} \tmp(u_{j,\varepsilon}) u_{j,\varepsilon} \psi \d x
- \int_{\Omega_r} |u_{j,\varepsilon}|^{q-2} u_{j,\varepsilon} \psi \d x\\
& - \lambda_{j,\varepsilon} \int_{\Omega_r} u_{j,\varepsilon} \psi \d x
- \int_{\Omega_r} |u_{j,\varepsilon}|^4 u_{j,\varepsilon} \psi \d x = o_\varepsilon(1),
\end{aligned}
\end{equation}
Hence, taking $\varphi = u_{j,\varepsilon}$, combining~\eqref{eq:2.5*}, we have
\[
  \lambda_{j,\varepsilon}b^2 = \int_{\Omega_r} |\nabla u_{j,\varepsilon}|^2 \d x
  + \int_{\Omega_r} \tmp(u_{j,\varepsilon}) u_{j,\varepsilon}^2 \d x
  - \int_{\Omega_r} |u_{j,\varepsilon}|^q \d x
  - \int_{\Omega_r} |u_{j,\varepsilon}|^6 \d x = o_\varepsilon (1),
\]
Then, we can know $\{\lambda_{j,\varepsilon}\}$ is bounded, thus $\exists \lambda_{j,0} \in \mathbb R$, such that $\lambda_{j,\varepsilon} \to \lambda_{j,0}$.
Inspired by~\cite{Benmlih2008,Figueiredo2019}, and combining with the result in~\eqref{eq:circ1.2},
we can fully simulate their methods of proof, then obtain when $u_{j,\varepsilon} \rightharpoonup u_{j,0}$ in $L^{6/5} (\Omega_r)$,
\begin{equation}\label{eq:circ2}
  \phi_{j,\varepsilon}(u_{j,\varepsilon}) \to \phi_{j,0}(u_{j, 0})\quad
  \text{in } H_0^1(\Omega_r), \quad \text{and} \quad
  \varepsilon \phi_{j,\varepsilon} (u_{j,\varepsilon}) \to 0
  \quad \text{in } W_0^{1,4}(\Omega_r).
\end{equation}
Note that the weak form of the second equation in system~\eqref{eq:1.1} is given by
\begin{equation}\label{eq:circ3}
  \int_{\Omega_r} \nabla \tmp \nabla \xi + \varepsilon^4 \int_{\Omega_r} |\nabla \tmp|^2 \nabla \tmp \nabla \xi
= \int_{\Omega_r} \xi u_{j,\varepsilon}^2 \d x,\quad
\forall \xi \in \mathbb H.
\end{equation}

Next, we will prove that as $\varepsilon \to 0^+$, equation \eqref{eq:circ3} converges to the weak form of the corresponding limit system
\begin{equation}\label{eq:circ4}
  \int_{\Omega_r} \nabla \phi_{j,0} \nabla \xi \d x = \int_{\Omega_r} \xi u_{j,0}^2 \d x.
\end{equation}
In fact, using~\eqref{eq:circ1.2}, we can easily obtain
\[
  \int_{\Omega_r} \nabla \tmp \nabla \xi \d x
  \to \int_{\Omega_r} \nabla \phi_{j,0} \nabla \xi \d x, \quad
  \int_{\Omega_r} \xi u_{j,\varepsilon}^2 \d x
  \to \int_{\Omega_r} \xi u_{j,0}^2 \d x.
\]
Using Holder's inequality and~\eqref{eq:circ2}, $\forall \xi \in \mathbb H$
\[
  \biggl|\varepsilon^4 \int_\Omega |\nabla \tmp|^2 \nabla\tmp \nabla \xi \d x\biggr|
  \leq \varepsilon^4 \|\nabla \tmp\|_4^3 \|\nabla \xi\|_4
= \varepsilon \cdot (\varepsilon \|\nabla \tmp\|_4)^3
\|\nabla \xi\|_4 \to 0,
\]
Let $\varepsilon \to 0$, then~\eqref{eq:circ4} satisfies, i.e., $\phi_{j,0}$ is the solution of the Poisson equation $-\Delta\phi = u_{j,0}$.

Similar to Lemma~\ref{lem:3.1*}, by using a measure represention Lemma, we can prove $u_{j,\varepsilon} \to u_{j,0}$ in $L^6(\Omega_r)$.
Next, we conclude that for $\forall \psi \in H_0^1(\Omega_r)$, we can get
\begin{equation}\label{eq:circ5}
  \begin{aligned}
    \int_{\Omega_r} \nabla u_{j,0} \nabla \psi \d x
    + \int_{\Omega_r} \phi_{j,0}(u_{j,0}) u_{j,0} \psi \d x
    & - \int_{\Omega_r} |u_{j,0}|^{q-2} u_{j,0} \psi\d x \\
    & - \lambda_{j,0} \int_{\Omega_r} u_{j,0} \psi \d x - \int_{\Omega_r} |u_{j,0}|^4 u_{j,0} \psi \d x = 0.
  \end{aligned}
\end{equation}
In fact, by standard argument, we have
\begin{gather}
  \int_{\Omega_r} \nabla u_{j,\varepsilon} \nabla \psi \d x
  \to \int_{\Omega_r} \nabla u_{j,0} \nabla \psi \d x,
  \label{eq:circ6}\\
  \lambda_{j,\varepsilon} \int_{\Omega_r} u_{j,\varepsilon} \psi \d x
  \to \lambda_{j,0} \int_{\Omega_r} u_{j,0} \psi \d x,
  \label{eq:circ7}\\
  \int_{\Omega_r} |u_{j,\varepsilon}|^4 u_{j,\varepsilon} \psi \d x
  \to \int_{\Omega_r} |u_{j,0}|^4 u_{j,0} \psi \d x.
  \label{eq:circ8}
\end{gather}
Due to $\tmp \to \phi_{j,0}$ in $L^6(\Omega_r)$,
$u_{j,\varepsilon} \to u_{j,0}$ in $L^{\frac{12}{5}}(\Omega_r)$,
using the H\"older inequality, we have
\begin{equation}\label{eq:circ9}
  \int_{\Omega_r} \tmp(u_{j,\varepsilon}) u_{j,\varepsilon} \psi \d x
  \to \int_{\Omega_r} \phi_{j,0}(u_{j,0}) u_{j,0} \psi \d x.
\end{equation}
Combining~\eqref{eq:circ6}-\eqref{eq:circ9}, we can know that~\eqref{eq:circ5}
satisfies, which means $(u_{j,0}, \phi_{j,0})$ is a pair of solution to~\eqref{eq:1.2} with $\lambda = \lambda_{j,0}$.

Finally, we consider the compactness.
In~\eqref{eq:4.1prime}, taking $\psi = u_{j,0}$ and $\psi = u_{j,\varepsilon}$
respectively, we can obtain
\begin{equation}\label{eq:circ10}
  \begin{aligned}
    \int_{\Omega_r} \nabla u_{j,\varepsilon} \nabla (u_{j,\varepsilon} - u_{j,0})\d x
    & + \int_{\Omega_r} \tmp(u_{j,\varepsilon}) u_{j,\varepsilon}
    (u_{j,\varepsilon} - u_{j,0}) \d x
    - \int_{\Omega_r} |u_{j,\varepsilon}|^{q-2} u_{j,\varepsilon}
    (u_{j,\varepsilon} - u_{j,0})\d x\\
    & - \lambda_{j,\varepsilon} \int_{\Omega_r} u_{j,\varepsilon} (u_{j,\varepsilon} - u_{j,0})\d x - \int_{\Omega_r} |u_{j,\varepsilon}|^4 u_{j,\varepsilon}
    (u_{j,\varepsilon} - u_{j,0}) \d x.
  \end{aligned}
\end{equation}
Using the same actions to~\eqref{eq:circ5}, then
\begin{equation}\label{eq:circ11}
  \begin{aligned}
    \int_{\Omega_r} \nabla u_{j,0} \nabla (u_{j,\varepsilon} - u_{j,0})\d x
    & + \int_{\Omega_r} \phi_{j,0}(u_{j,0})u_{j,0} (u_{j,\varepsilon} - u_{j,0})\d x\\
    & - \int_{\Omega_r} |u_{j,0}|^{q-2} u_{j,0}(u_{j,\varepsilon}-u_{j,0})\d x\\
    & -\lambda_{j,0} \int_{\Omega_r} u_{j,0} (u_{j,\varepsilon} - u_{j,0})\d x
    - \int_{\Omega_r} |u_{j,0}|^4 u_{j,0} (u_{j,\varepsilon} - u_{j,0}) \d x.
  \end{aligned}
\end{equation}
Using H\"older inequality, we have
\begin{equation}\label{eq:circ12}
  \int_{\Omega_r} |u_{j,\varepsilon}|^4 u_{j,\varepsilon} (u_{j,\varepsilon} - u_{j,0})\d x
  \leq
  \|u_{j,\varepsilon}\|_6^4 \|u_{j,\varepsilon}\|_6
  \|u_{j,\varepsilon} - u_{j,0}\|_6 \to 0
\end{equation}
and
\begin{equation}\label{eq:circ13}
  \int_{\Omega_r} \tmp(u_{j,\varepsilon}) u_{j,\varepsilon}
  (u_{j,\varepsilon} - u_{j,0}) \d x
  \leq \|\tmp(u_{j,\varepsilon})\|_6\|u_{j,\varepsilon}\|_{\frac32}
  \|u_{j,\varepsilon} - u_{j,0}\|_6 \to 0.
\end{equation}
Repeat the steps above, we can get
\[
  \int_{\Omega_r} |u_{j,0}|^4 u_{j,0} (u_{j,\varepsilon} - u_{j,0}) \d x \to 0
  \quad \text{and} \quad
  \int_{\Omega_r} \phi_{j,0}(u_{j,0}) u_{j,0} (u_{j,\varepsilon}-u_{j,0}) \d x \to 0.
\]
Combining~\eqref{eq:circ10}-\eqref{eq:circ13}, we can get
\[
  \int_{\Omega_r} |\nabla(u_{j,\varepsilon} - u_{j,0})|^2 \d x + \lambda_{j,0}
  \int_{\Omega_r} |u_{j,\varepsilon} - u_{j,0}|^2 \d x = o_\varepsilon(1).
\]
Since $u_{j,\varepsilon} \to u_{j,0}$ in $L^2(\Omega_r)$, then
\[
  \int_{\Omega_r} |\nabla (u_{j,\varepsilon} - u_{j,0})|^2 \d x \to 0.
\]
Hence, $u_{j,\varepsilon} \to u_{j,0}$ in $H_0^1(\Omega_r)$.
The Theorem~\ref{thm:1.2} is proved.

\section*{Declarations}

\paragraph{Acknowledgements}
Li Wang was supported by National Natural Science Foundation of China (Grant Nos. 12161038, 12301584), Jiangxi Provincial Natural Science Foundation (Grant No. 20232BAB201009), Science and Technology Project of Jiangxi Provincial Department of Education (Grant No. GJJ2400901).

\paragraph{Author Contributions}
All authors have accepted reponsibility for the entire content of this version of the manuscript and consented to its submission to the journal, reviewed all the results and approved the final version of the manuscript.
Li Chen proposed research ideas and writing original draft preparation.
Li Wang provided this idea for the study and led the implementation review and revision of the manuscript. All authors have read and agreed to the published version of the manuscript.

\paragraph{Conflict of interest}
The authors declare that they have no conflict of interest.

\end{document}